\documentclass[12pt,a4paper]{amsart}
\usepackage{amssymb, amsthm, amsmath}

\theoremstyle{definition}
\newtheorem{thm}{Theorem}[section]
\newtheorem{defn}[thm]{Definition}
\newtheorem{prop}[thm]{Proposition}
\newtheorem{lemma}[thm]{Lemma}
\newtheorem{cor}[thm]{Corollary}
\newtheorem{example}[thm]{Example}

\newcommand{\Gal}[1]{\mbox{Gal}(#1)}
\newcommand{\cent}[2]{\mbox{Cent}_{#1}\left(#2\right)}
\renewcommand{\bar}{\overline}

\newcommand{\B}{\mathfrak B}
\newcommand{\OK}{{\mathfrak O}_{K}}
\newcommand{\OL}{{\mathfrak O}_{L}}
\newcommand{\A}{{\mathfrak A}}
\newcommand{\AH}{{\mathfrak A}_{H}}
\newcommand{\perm}[1]{\mbox{Perm}(#1)}

\newcommand{\ndivides}{ \hspace{0mm} \nmid }

\def \p  {{\mathfrak p}}
\def \OKp {{\mathfrak O}_{K,\p}}

\begin{document}

\title[Commuting Hopf-Galois Structures]{Commuting Hopf-Galois Structures on a Separable Extension}

\author{Paul J. Truman}

\address{School of Computing and Mathematics \\ Keele University \\ Staffordshire \\ ST5 5BG \\ UK}

\email{P.J.Truman@Keele.ac.uk}

\subjclass[2000]{Primary 11R33; Secondary 11S23}

\keywords{Hopf-Galois structure, Hopf-Galois module theory, Galois module structure, Associated order}

\begin{abstract}
Let $ L/K $ be a finite separable extension of local or global fields in any characteristic, let $ H_{1}, H_{2} $ be two Hopf algebras giving Hopf-Galois structures on the extension, and suppose that the actions of $ H_{1}, H_{2} $ on $ L $ commute. We show that a fractional ideal $ \B $ of $ L $ is free over its associated order in $ H_{1} $ if and only if it is free over its associated order in $ H_{2} $. We also study which properties these associated orders share. 
\end{abstract}

\maketitle

\section{Introduction and Statement of Results} \label{section_introduction}

This paper is a sequel to \cite{PJT_Canonical}. In the introduction to that paper, we described one of the ways in which Hopf-Galois module theory generalizes the classical Galois module theory of algebraic integers. Given a finite Galois extension of local or global fields $ L/K $ with Galois group $ G $, the group algebra $ K[G] $ is a Hopf algebra, and its action on $ L $ is an example of a {\em Hopf-Galois structure} on the extension $ L/K $ (see \cite[Definition 2.7]{ChTwe}). However, the extension may admit a number of other Hopf-Galois structures, and each of these provides a different context in which we can ask module theoretic questions about the extension and its fractional ideals. If $ H $ is a Hopf algebra giving a Hopf-Galois structure on $ L/K $ then $ L $ is a free $ H $-module of rank $ 1 $, (\cite[Proposition 2.15]{ChTwe} and if $ \B $ is a fractional ideal of $ L $ then we can define the associated order of $ \B $ in $ H $:
\[ \AH(\B) = \{ h \in H \mid h \cdot x \in \B \mbox{ for all } x \in \B \}, \]
and compare the structure of $ \B $ as a module over its associated orders in the various Hopf algebras. The most interesting case is $ \B = \OL $, the ring of algebraic integers of $ L $, and there exist Galois extensions $ L/K $ of $ p $-adic fields for which $ \OL $ is not free over its associated order in the group algebra $ K[G] $ but is free over its associated order in some Hopf algebra giving a different Hopf-Galois structure on the extension \cite{By97a}. On the other hand, there exist extensions admitting multiple Hopf-Galois structures for which $ \OL $ is free over its associated order in each of the corresponding Hopf algebras \cite{PJT_Noether}.  
\\ \\
Another way in which Hopf-Galois module theory generalizes classical Galois module theory is that a finite separable, but non-normal, extension of local or global fields may admit Hopf-Galois structures; here of course the techniques of classical Galois module theory are not available, but these Hopf-Galois structures allow us to study such extensions and their fractional ideals as described above.
\\ \\
A theorem of Greither and Pareigis (\cite[Theorem 3.1]{Greither_Pareigis} or \cite[Theorem 6.8]{ChTwe}) enumerates and describes all of the Hopf-Galois structures admitted by a given finite separable extension of fields $ L/K $, as follows: let $ E/K $ be the Galois closure of $ L/K $, $ G=\Gal{E/K} $, and $ G_{L} = \Gal{E/L} $. Now let $ X $ be the left coset space of $ G_{L} $ in $ G $, and $ B = \perm{X} $. Then the Hopf-Galois structures on $ L/K $ are in bijective correspondence with the regular subgroups of $ B $ normalized the image of $ G $ under the embedding $ \lambda : G \rightarrow B $ defined by $ \lambda(\sigma) gG_{L} = \sigma g G_{L} $ . (Recall that a subgroup of $ B $ is called {\em regular} if it has the same order as $ X $ and is transitive on $ X $.) Furthermore, the Hopf algebra corresponding to a regular subgroup $ N $ is $ E[N]^{G} $, where $ G $ acts on $ E $ as Galois automorphisms and on $ N $ by conjugation via the embedding $ \lambda $:
\begin{equation} \label{eqn_GP_conj}
 \,^{g}\eta = \lambda(g) \eta \lambda (g^{-1}) \mbox{ for all } g \in G, \; \eta \in N. 
\end{equation}
Finally, such a Hopf algebra acts on $ L $ by
\begin{equation}\label{eqn_GP_action} 
\left( \sum_{\eta \in N} c_{\eta} \eta \right) \cdot x = \sum_{\eta \in N} c_{\eta} \eta^{-1}(\bar{1})[x] \hspace{5mm} (c_{\eta} \in E, \; x \in L),
\end{equation}
where $ \bar{1} $ denotes the coset $ G_{L} $ in $ X $. 
\\ \\
If $ L/K $ is a Galois extension then in the notation above $ X = G $ and $ \lambda $ is the left regular embedding of $ G $ into $ \perm{G} $. In this case, one example of a regular subgroup of $ \perm{G} $ normalized by $ \lambda(G) $ is $ \rho(G) $, the image of $ G $ under the right regular embedding. Since $ \lambda(G) $ centralizes $ \rho(G) $ inside $ \perm{G} $, the Hopf-Galois structure corresponding to the regular subgroup $ \rho(G) $ has Hopf algebra isomorphic to $ K[G] $, and from equation \eqref{eqn_GP_action} we recover the usual action of $ K[G] $ on $ L $  (see \cite[Proposition 6.10]{ChTwe}). We call this Hopf-Galois structure the {\em classical structure}. If $ G $ is abelian then in fact $ \lambda(G) = \rho(G) $, but if $ G $ is nonabelian then they are distinct, and $ \lambda(G) $ is another example of a regular subgroup of $ \perm{G} $ normalized by $ \lambda(G) $, and therefore corresponds to a different Hopf-Galois structure on $ L/K $.  In \cite{PJT_Canonical} we called this the {\em canonical nonclassical structure}, and denoted the corresponding Hopf algebra $ L[\lambda(G)]^{G} $ by $ H_{\lambda} $. We proved two theorems concerning the relationship between the classical structure and the canonical nonclassical structure: an element $ x \in L $ generates $ L $ as a $ K[G] $ module if and only if it generates $ L $ as an $ H_{\lambda} $-module (\cite[Theorem 2.1]{PJT_Canonical}), and, assuming that $ \OK $ is a principal ideal domain, an ambiguous ideal of $ L $ is free over its associated order in $ K[G] $ if and only if it is free over its associated order in $ H_{\lambda} $ (\cite[Theorem 1.1]{PJT_Canonical} and its correction,  \cite{PJT_Canonical_Correction}). The proofs of these results exploited the fact that 
\[ h \cdot g(x) = g(h \cdot x) \mbox{ for all } h \in H_{\lambda}, \; g \in G, \; x \in L \]
(see \cite[Lemma 3.3]{PJT_Canonical}). In this paper we generalize the results of \cite{PJT_Canonical} to pairs of Hopf-Galois structures on a given separable (but not necessarily Galois) extension whose actions on $ L $ {\em commute}, by which we mean that 
\[ h_{1} \cdot ( h_{2} \cdot x ) = h_{2} \cdot ( h_{1} \cdot x ) \mbox{ for all } h_{1} \in H_{1}, \; h_{2} \in H_{2}, \; x \in L. \]
Furthermore, our results apply to all fractional ideals of $ L $ (not just to ambiguous ideals), and we can remove the hypothesis that $ \OK $ is a principal ideal domain.
\\ \\
Our results are as follows:

\begin{thm} \label{thm_NB_generators_intro}
Let $ L/K $ be a finite separable extension of fields and let $ H_{1},H_{2} $ be two Hopf algebras giving Hopf-Galois structures on $ L/K $ whose actions commute. Then an element $ x \in L $ generates $ L $ as an $ H_{1} $-module if and only if it generates $ L $ as an $ H_{2} $-module. 
\end{thm}

\begin{thm} \label{thm_main_thm_intro}
Retain the notation and hypotheses of Theorem \ref{thm_NB_generators_intro}, and assume that $ L/K $ is an extension of local or global fields (in any characteristic). Let $ \B $ be a fractional ideal of $ L $, and let $ \A_{1} $, $ \A_{2} $ be the associated orders of $ \B $ in $ H_{1} $, $ H_{2} $ respectively. Then $ \B $ is a free $ \A_{1} $-module if and only if it is a free $ \A_{2} $-module. 
\end{thm}

\begin{cor}
If $ L/K $ is an extension of global fields then $ \B $ is a locally free $ \A_{1} $-module if and only if it is a locally free $ \A_{2} $-module. 
\end{cor}
\begin{proof}
The proof of Theorem \ref{thm_main_thm_intro} does not depend on the fact that $ L $ is a field, so we may replace $ L $ with its completion at some prime $ \p $ of $ \OK $ (a Galois algebra). To say that $ \B $ is locally free over $ \A_{1} $ is to say that for each prime $ \p $ of $ \OK $, we have $ \B_{\p} = \OKp \otimes_{\OK} \B $ is a free $ \OKp \otimes_{\OK} \A_{1} $-module. By Theorem \ref{thm_main_thm_intro} this is equivalent to saying that for each prime $ \B_{\p} $ is a free $ \OKp \otimes_{\OK} \A_{2} $-module, which is saying that $ \B $ is a locally free $ \A_{2} $-module. 
\end{proof}

In section \ref{section_commuting_structures} we characterize Hopf-Galois structures on a given separable extension whose actions commute, using the theorem of Greither and Pareigis. In section \ref{section_NBG_module_structure} we generalize some of the results of \cite[Section 2]{PJT_Canonical}, and use these to prove Theorems \ref{thm_NB_generators_intro} and \ref{thm_main_thm_intro}. Finally, in section \ref{section_shared_properties} we ask, under the hypotheses of Theorem \ref{thm_main_thm_intro}, which algebraic properties the associated orders $ \A_{1} $ and $ \A_{2} $ share. We show that $ \A_{1} $ is a maximal order in $ H_{1} $ if and only if $ \A_{2} $ is a maximal order in $ H_{2} $ (Proposition \ref{prop_maximal_orders}), but that it is possible for $ \A_{1} $ to be a Hopf order but for $ \A_{2} $ not to have this property (Example \ref{example_hopf}).
\\ \\ 
I would like to express my gratitude to Prof. Nigel Byott, who first suggested to me that it might be possible to generalize the results of \cite{PJT_Canonical} in this way. 

\section{Commuting Hopf-Galois Structures} \label{section_commuting_structures}

Let $ L/K $ be a finite separable extension of fields. In this section we use the theorem of Greither and Pareigis to characterize Hopf-Galois structure on $ L/K $ whose actions on $ L $ commute.  Although we are principally interested in extensions of local or global fields, we do not impose this hypothesis since the the results of this section are valid more generally. We retain the hypotheses and notation used in the theorem of Greither and Pareigis, as described in the introduction: $ E/K $ is the Galois closure of $ L/K $, $ G=\Gal{E/K} $, $ G_{L} = \Gal{E/L} $, $ X = \{ gG_{L} \mid g \in G \} $, $ B = \perm{X} $, and $ \lambda : G \rightarrow B $ is defined by $ \lambda(\sigma)gG_{L} = \sigma g G_{L} $. We shall always think of a Hopf algebra produced by the theorem of Greither and Pareigis as acting via equation 
\eqref{eqn_GP_action}; hence we may refer without ambiguity to {\em the} Hopf-Galois structure given by a particular Hopf algebra. In the introduction we noted that if $ L/K $ is a Galois extension with Galois group $ G $ then $ \lambda(G) $ and $ \rho(G) $ are both regular subgroups of $ B $ normalized by $ \lambda(G) $, and therefore correspond to Hopf-Galois structures on $ L/K $. (If $ G $ is abelian then these subgroups coincide, and so both yield the same Hopf-Galois structure.) We also noted that $ \lambda(G) $ centralizes $ \rho(G) $ inside $ B $; in fact $ \lambda(G) = \cent{B}{\rho(G)} $. This relationship between $ \lambda(G) $ and $ \rho(G) $ in the Galois case is a particular example of a more general phenomenon:  

\begin{lemma} \label{lem_N_N_prime}
If $ N $ is a regular subgroup of $ B $ which is normalized by $ \lambda(G) $, then so is $ N^{\prime} = \cent{B}{N} $. 
\end{lemma}
\begin{proof}
The subgroup $ N^{\prime} $ is explicitly constructed in \cite[Lemma 2.4.2]{Greither_Pareigis} as follows: since $ N $ is regular on $ X $, for each coset $ \bar{g} \in X $ there is a unique element $ \mu_{\bar{g}} \in N $ such that $ \mu_{\bar{g}}(\bar{1})=\bar{g} $. For each $ \eta \in N $ define $ \phi_{\eta} \in B $ by $ \phi_{\eta}(\bar{g})=\mu_{\bar{g}}(\eta(\bar{1})) $; then $ N^{\prime} = \{ \phi_{\eta} \mid \eta \in N \} $. From this it is easy to verify that $ N^{\prime} $ is a regular subgroup of $ B $. The proof that $ N^{\prime} $ is normalized by $ \lambda(G) $ appears as part of the proof of \cite[Theorem 2.5]{Greither_Pareigis}. These facts are also established in \cite[Proposition 3.2, Corollary 3.9]{Kohl_4p}.
\end{proof}

\begin{defn}
If $ N $ is a regular subgroup of $ X $ normalized by $ \lambda(G) $, and $ H $ is the Hopf algebra giving the corresponding Hopf-Galois structure, let $ N^{\prime} = \cent{B}{N} $ and $ H^{\prime} = E[N^{\prime}]^{G} $. 
\end{defn}

We record some properties of the relationships between the regular subgroups $ N $ and $ N^{\prime} $ and between the Hopf algebras $ H $ and $ H^{\prime} $:

\begin{lemma} \label{lemma_properties_of_N_opp}
Let $ N $ be  a regular subgroup of $ B $ which is normalized by $ \lambda(G) $. Then:
\begin{enumerate}
\item We may identify $ N^{\prime} $ with $ N^{opp} $,
\item $ N^{\prime} \cong N $,
\item $ \left( N^{\prime} \right)^{\prime} = N $,
\item $ \left( H^{\prime} \right)^{\prime} = H $,
\item $ N = N^{\prime} $ if and only if $ N $ is abelian,
\item The Hopf-Galois structures given by $ H $ and $ H^{\prime} $ coincide if and only if $ H $ is commutative. 
\end{enumerate}
\end{lemma}
\begin{proof}
\begin{enumerate}
\item[]
\item See \cite[Lemma 2.4.2]{Greither_Pareigis}.
\item  In the notation of the proof of Lemma \ref{lem_N_N_prime}, the map $ \eta \mapsto \phi_{\eta^{-1}} $ is an isomorphism from $ N $ to $ N^{\prime} $. 
\item See \cite[Lemma 3.5]{Kohl_4p}.
\item This follows from part (iii) and the definition of $ H^{\prime} $.
\item If $ N  $ is abelian then $ N \subseteq N^{\prime} $, but these groups have the same order, so in fact $ N = N^{\prime} $. For the converse, we recall from \cite[Proposition 3.3]{Kohl_4p} that $ N \cap N^{\prime} = Z(N) $, the center of $ N $, so if $ N=N^{\prime} $ then $ N $ is abelian. 
\item This follows from part (v) and the definition of $ H^{\prime} $.
\end{enumerate}
\end{proof}

Since we may identify $ N^{\prime} $ with $ N^{opp} $, we may identify the group algebra $ E[N^{\prime}] $ with the opposite ring $ E[N]^{opp} $. However, we may not identify the Hopf algebra $ H^{\prime}$ with the opposite ring $ H^{opp} $, since the action of $ G $ on $ N^{opp} $ (see equation \eqref{eqn_GP_conj}) may not be the same as its action on $ N $. This is already apparent in the case that $ L/K $ is Galois, $ N = \rho(G) $ and $ N^{opp} = \lambda(G) $. Since $ \lambda(G) $ is the centralizer of $ \rho(G) $ in $ B $, the action of $ G $ on $ \rho(G) $ is trivial, and so $ L[\rho(G)]^{G} = L^{G}[\rho(G)] \cong K[G] $. However, the action of $ G $ on $ \lambda(G) $ is not trivial (the orbits are the conjugacy classes), and so $ L[\lambda(G)]^{G} \neq K[\lambda(G)] $. Therefore in this case we may not identify $ L[\lambda(G)]^{G} $ with $ K[\rho(G)]^{opp} $. 
\\ \\
We shall show that if $ H $ is a Hopf algebra giving a Hopf-Galois structure on $ L/K $ then the actions of $ H $ and $ H^{\prime} $ on $ L $ commute (Proposition \ref{prop_actions_commute}) and, conversely, that if $ H_{1} $ and $ H_{2} $ are two Hopf algebras giving Hopf-Galois structure on $ L/K $ whose actions on $ L $ commute then $ H_{2} = H_{1}^{\prime} $ (Proposition \ref{prop_actions_commute_converse}). To do this, we recall some notation employed in the proof of the theorem of Greither and Pareigis, as detailed in \cite[\S 6]{ChTwe}. Let $ M = \mbox{Map}(X,E) $, and let $ \{ u_{\bar{g}} \mid \bar{g}\in X \} $ be an $ E $-basis of mutually orthogonal idempotents (where $ \bar{g} $ denotes the left coset $ gG_{L} \in X $). That is:
\[ u_{\bar{g}}(\bar{\sigma}) = \delta_{\bar{g},\bar{\sigma}} \mbox{ for all } \bar{g},\bar{\sigma} \in X. \]
It can be shown \cite[Theorem 6.3]{ChTwe} that the $ E $-Hopf algebras giving Hopf-Galois structures on the extension of rings $ M / E $ are precisely the group algebras $ E[N] $ of regular subgroups $ N $ of $ \perm{G} $, where the group $ N $ acts on $ M $ by permuting the subscripts of the idempotents $ u_{\bar{g}} $:
\[ \eta \cdot u_{\bar{g}} = u_{\eta(\bar{g})} \mbox{ for any } \eta \in N \mbox{ and } \bar{g} \in X. \]
If in addition $ N $ is normalized by $ \lambda(G) $ then the group $ G $ acts on $ E[N] $ by acting on $ E $ as Galois automorphisms and on $ N $ by conjugation via the image of the embedding $ \lambda $ into $ \perm{X} $. It also acts on $ M $ by acting on $ E $ as Galois automorphisms and on the idempotents $ u_{\bar{g}} $ by left translation of the subscripts. Now by Galois descent we obtain that $ E[N]^{G} $ is a $ K $-Hopf algebra giving a Hopf-Galois structure on the extension of rings $ M^{G} / K $. Note also that $ E \otimes_{K} E[N]^{G} = E[N] $ and $ E \otimes_{K} M^{G} = M $. Finally, we may identify $ L $ with the fixed ring $ M^{G} $ via the $ K $-algebra isomorphism $ L \xrightarrow{\sim} M^{G} $ defined by
\begin{equation} \label{eqn_L_MG_isomorphism}
x \mapsto f_{x} =  \sum_{\bar{g} \in X} g(x) u_{\bar{g}} \mbox{ for all } x \in L. 
\end{equation}
(For $ \bar{g} \in X $, the element $ g(x) $ is well defined since $ g_{0}(t)=t $ for all $ g_{0} \in G_{L} $ and $ t \in L $). Thus $ E[N]^{G} $ gives a Hopf-Galois structure on $ L/K $, with the action of $ E[N]^{G} $ on $ L $ as given in equation (\ref{eqn_GP_action}). 
\\ \\
With this notation to hand, we show that if $ H $ is a Hopf algebra giving a Hopf-Galois structure on $ L/K $ then the actions of $ H $ and $ H^{\prime} $ on $ L $ commute. We then show that any pair of Hopf algebras giving Hopf-Galois structures on $ L/K $ whose actions on $ L $ commute arises in this way. 

\begin{prop} \label{prop_actions_commute}
Let $ H $ be a Hopf algebra giving a Hopf-Galois structure on $ L/K $, Then the actions of $ H $ and $ H^{\prime} $ on $ L $ commute. 
\end{prop}
\begin{proof}
Since $ \eta^{\prime} \eta = \eta \eta^{\prime} $ for all $ \eta \in N $ and $ \eta^{\prime} \in N^{\prime} $, we have $ z^{\prime} \cdot ( z \cdot f ) = z \cdot ( z^{\prime} \cdot f ) $ for all $ z \in E[N] $, $ z^{\prime} \in E[N^{\prime}] $ and $ f \in M $. Therefore it is certainly true that $ h^{\prime} \cdot ( h \cdot f ) = h \cdot ( h^{\prime} \cdot f ) $ for all $ h \in E[N]^{G} $, $ h^{\prime} \in E[N^{\prime}]^{G} $ and $ f \in M^{G} $, and so the actions of $ H $ and $ H^{\prime} $ on $ L $ commute. 
\end{proof}

\begin{prop} \label{prop_actions_commute_converse}
If $ H_{1} = E[N_{1}]^{G} $ and $ H_{2} = E[N_{2}]^{G} $ are two Hopf algebras giving Hopf-Galois structures on the extension $ L/K $ whose actions on $ L $ commute, then we have $ N_{2} = N_{1}^{\prime} $.
\end{prop}
\begin{proof}
Since the actions of $ H_{1}, H_{2} $ on $ L $ commute, the actions of $ H_{1}, H_{2} $ on $ M^{G} $ commute. Since $ E \otimes_{K} M^{G} = M $ and $ E \otimes_{K} H_{i} = E[N_{i}] $ for $ i=1,2 $, this implies that the actions of $ E[N_{1}], E[N_{2}] $ on $ M $ commute. Therefore for all $ \eta_{1} \in N_{1} $, $ \eta_{2} \in N_{2} $ we have
\begin{eqnarray*}
&& \eta_{1} \eta_{2} u_{\bar{g}}  = \eta_{2} \eta_{1} u_{\bar{g}}  \mbox{ for all } \bar{g} \in X \\
&\Rightarrow & \eta_{1}^{-1} \eta_{2}^{-1} \eta_{1} \eta_{2} \bar{g} =\bar{g}  \mbox{ for all } \bar{g} \in X \\
&\Rightarrow & \eta_{1}^{-1} \eta_{2}^{-1} \eta_{1} \eta_{2}  = 1  \\
&\Rightarrow &  \eta_{1} \eta_{2}  = \eta_{2} \eta_{1}. 
\end{eqnarray*}
Therefore $ N_{2} \subset N_{1}^{\prime} $. But $ N_{2} $ and $ N_{1}^{\prime} $ have the same order, so in fact they must be equal. 
\end{proof}

\section{Normal Basis Generators and Integral Module Structure} \label{section_NBG_module_structure}

We continue to assume that $ L/K $ is a finite separable extension of fields, and retain the notation established in section \ref{section_introduction} concerning the statement of the theorem of Greither and Pareigis, as well as the elements of its proof that we used in section \ref{section_commuting_structures}. In this section we prove Theorems \ref{thm_NB_generators_intro} and \ref{thm_main_thm_intro}, which are our main results. We require slight generalizations of two lemmas from \cite{PJT_Canonical}:

\begin{lemma} \label{lem_fixed_generators}
Let $ N $ be a regular subgroup of $ \perm{X} $ normalized by $ \lambda(G) $, so that $ E[N]^{G} $ is a $ K $-Hopf algebra giving a Hopf-Galois structure on $ M^{G} / K $. An element $ f \in M^{G} $ is an $ E[N]^{G} $-generator of $ M^{G} $ if and only if it is an $ E[N] $-generator of $ M $. 
\end{lemma}
\begin{proof}
This is a slight generalization of \cite[Lemma 2.2]{PJT_Canonical}, with essentially the same proof. 
\end{proof}

\begin{lemma} \label{lem_transition_matrix}
Fix orderings of the set $ X $ and the group $ N $. For $ x \in L $, the element $ f_{x} $ is an $ E[N] $-generator of $ M $ if and only if the matrix 
\[ T_{N}(x) = ( \eta(\bar{g})[x] )_{\eta \in N, \; \bar{g} \in X} \]
is nonsingular. 
\end{lemma}
\begin{proof}
This is a slight generalization of \cite[Lemma 2.3]{PJT_Canonical}, with essentially the same proof. We note that although the definition of the matrix  $ T_{N}(x) $ depends on the orderings of $ G $ and $ N $, the question of whether it is nonsingular does not. 
\end{proof}

We now use the characterization of commuting Hopf-Galois structures from section \ref{section_commuting_structures} to prove Theorem \ref{thm_NB_generators_intro}: which generalizes \cite[Theorem 1.2]{PJT_Canonical}:

\begin{proof}[Proof of Theorem \ref{thm_NB_generators_intro}.]
By the theorem of Greither and Pareigis, $  H_{1} = E[N]^{G} $ for some regular subgroup $ N $ of $ X $ normalized by $ \lambda(G) $, and by Proposition  \ref{prop_actions_commute_converse} we have $ H_{2} = H_{1}^{\prime} = E[N^{\prime}]^{G} $. Thus by Lemmas \ref{lem_fixed_generators} and \ref{lem_transition_matrix} it is sufficient to show that $ T_{N}(x) $ is nonsingular if and only if $ T_{N^{\prime}}(x) $ is nonsingular. We have:
\begin{eqnarray*}
\det( T_{N}(x)) & = & \det \left( ( \eta(\bar{g})[x] )_{\eta \in N, \; \bar{g} \in X} \right) \\
& = & \det \left( ( \eta(\eta^{\prime}(\bar{1_{G}}))[x] )_{\eta \in N, \; \eta^{\prime} \in N^{\prime}} \right) \mbox{ since } N^{\prime} \mbox{ is regular on } X \\
& = & \det \left( ( \eta^{\prime}(\eta(\bar{1_{G}}))[x] )_{\eta \in N, \; \eta^{\prime} \in N^{\prime}} \right) \mbox{ since } N, N ^{\prime} \mbox{ commute inside } B \\
& = & \det \left( ( \eta^{\prime}(\bar{g})[x] )_{\eta^{\prime} \in N^{\prime}, \; \bar{g} \in X}  \right) \mbox{ since } N \mbox{ is regular on } X \\
& = & \det( T_{N^{\prime}}(x))
\end{eqnarray*}
Therefore $ T_{N}(x) $ is nonsingular if and only if $ T_{N^{\prime}}(x) $ is nonsingular, which completes the proof. 
\end{proof}

The results we have established so far are valid for any finite separable extension of fields. We now turn to questions of integral Hopf-Galois module structure in a finite separable extension of local or global fields $ L/K $. Note, though, that we make no restriction on the characteristic of $ K $. Using the notation established in section \ref{section_commuting_structures}, we will prove Theorem \ref{thm_main_thm_intro}, which generalizes \cite[Theorem 1.1]{PJT_Canonical}.  First we note that if $ H $ is a Hopf algebra giving a Hopf-Galois structure on $ L/K $ and $ \B $ is a fractional ideal of $ L $ then $ \B $ has an associated order in $ H $: since $ L $ is a free $ H $-module of rank $ 1 $ (\cite[Proposition 2.15]{ChTwe}), we can identify $ \B $ with a full $ \OK $-lattice in $ H $, and the left multiplier ring of this lattice is an order in $ H $ \cite[Chapter 2, \S 8]{Reiner_MO}, which we can identify with the associated order of $ \B $ in $ H $. Thus the assumption in \cite[Theorem 1.1]{PJT_Canonical} that the fractional ideal $ \B $ is an ambiguous ideal of $ L $ is superfluous. 

\begin{proof}[Proof of Theorem \ref{thm_main_thm_intro}.]
Let $ H_{1}=H $; then by Proposition  \ref{prop_actions_commute_converse} we have $ H_{2} = H^{\prime} $. Suppose that $ x \in \B $ generates $ \B $ as an $ \AH $-module. Then $ x $ generates $ L $ as an $ H $-module, and so by Theorem \ref{thm_NB_generators_intro} it generates $ L $ as an $ H^{\prime} $-module. Therefore for each $ a \in \AH $ we may define $ z_{a} \in H^{\prime} $ by $ z_{a} \cdot x = a \cdot x $. We claim that
\[ \AH^{\prime} = \{ z_{a} \mid a \in \AH \}. \]
First we show that $ z_{a} \in \AH^{\prime} $ for each $ a \in \AH $. Let $ a \in \AH $ and $ b \in \B $. Since $ \B $ is a free $ \AH $-module, there exists a unique $ w \in \AH $ such that $ b = w \cdot x $. Now we have:
\begin{eqnarray*}
z_{a} \cdot b & = & z_{a} \cdot ( w \cdot x ) \\
& = & w \cdot ( z_{a} \cdot x ) \mbox{ since the actions of } H, H^{\prime} \mbox{ on } L \mbox{ commute}\\
& = & w \cdot ( a \cdot x ),
\end{eqnarray*} 
and this lies in $ \B $ since $ a,w \in \AH $. Therefore $ z_{a} \in \AH^{\prime} $. 
\\ \\
On the other hand, if $ z \in \AH^{\prime} $ then $ z \cdot x = z_{a} \cdot x $ for some $ a \in  \AH $, and this implies that $ z = z_{a} $ because $ x $ generates $ L $ as a free $ H^{\prime} $-module. Therefore $ \AH^{\prime} = \{ z_{a} \mid a \in \AH \} $, as claimed. 
\\ \\
Since $ (H^{\prime})^{\prime} = H $, the converse statement follows by interchanging the roles of $ H, H^{\prime} $ in the argument above.
\end{proof}

In \cite{PJT_Canonical}, we split the proof of Theorem 1.1 into two propositions: Proposition 4.1 and Proposition 4.2. In the proofs of those propositions we assumed that certain $ \OK $ orders in certain $ K $-algebras possessed $ \OK $-bases. Although this is automatically true if $ L/K $ is an extension of local fields, it need not be true if $ L/K $ is an extension of global fields. In order for these proofs to be correct as written, we therefore issued a correction (see \cite{PJT_Canonical_Correction}) adding the hypothesis that $ \OK $ must be a principal ideal domain. In the proof of Theorem \ref{thm_main_thm_intro} above, we made no such assumptions about $ \OK $-bases, and so this hypothesis is no longer needed. In particular,  \cite[Theorem 1.1]{PJT_Canonical} is valid as stated. 

\section{Shared Properties of Associated Orders} \label{section_shared_properties}

In this section we adopt the hypotheses of Theorem \ref{thm_main_thm_intro}: $ L/K $ is a finite separable extension of local or global fields, $ \B $ is a fractional ideal of $ L $, $ H $ is a Hopf algebra giving a Hopf-Galois structure on $ L/K $, and $ H^{\prime} $ is the Hopf algebra whose action on $ L $ commutes with that of $ H $. We have seen that $ \B $ is free over $ \AH $ if and only if it is free over $ \AH^{\prime} $, but we might wonder whether these orders share any algebraic properties. If $ K $ has characteristic zero, then the question of whether $ \AH $ is a Hopf order in $ H $ is particularly interesting, since in this case $ \OL $ is locally free over $ \AH $ by a theorem of Childs \cite[Theorem 13.4]{ChTwe}. However, there are examples of extensions for which $ \AH $ is a Hopf order but $ \AH^{\prime} $ is not: 

\begin{example} \label{example_hopf}
Let $ p,q,r $ be prime numbers with $ r \equiv 1 \pmod{q} $ and $ p \ndivides qr $, and let $ L/K $ be a tamely ramified Galois extension of $ p $-adic fields whose Galois  group $ G $ is isomorphic to the metacyclic group of order $ qr $. For a fixed integer $ d $ having order $ q $ modulo $ r $, we may write
\[ G = \langle \sigma, \tau \mid \sigma^{r} = \tau^{q} = 1, \; \tau \sigma = \sigma^{d} \tau \rangle. \] 
The extension $ L/K $ is Galois with nonabelian Galois group, so it admits the classical Hopf-Galois structure, with Hopf algebra $ K[G] $, and the canonical nonclassical structure, with Hopf algebra $ H_{\lambda} = L[\lambda(G)]^{G} $. Since $ L/K $ is tamely ramified, Noether's Theorem implies that $ \OL $ is a free module over $ \OK[G] $, which is a Hopf order in $ K[G] $. By Theorem \ref{thm_main_thm_intro}, $ \OL $ is free over its associated order $ \A_{\lambda} $ in $ H_{\lambda} $, but we shall show that this is not a Hopf order in $ H_{\lambda} $. 
\\ \\
By \cite[Proposition 2.5]{PJT_Noether}, we have $ \OL[\lambda(G)]^{G} \subseteq \A_{\lambda} $. We begin by showing that $ \OL[\lambda(G)]^{G} $ is not a Hopf order in $ H_{\lambda} $. By \cite[Corollary 2.2]{PJT_Integral} $ \OL[\lambda(G)]^{G} $ is a Hopf order in $ H_{\lambda} $ if and only if the inertia subgroup of $ G $ is contained in the center of $ G $. Since $ G $ is nonabelian, the extension $ L/K $ is neither unramified nor totally ramified, and so the inertia subgroup must be the unique nontrivial normal subgroup of $ G $, which is $ \langle \sigma \rangle $. But it is easy to see that $ G $ has trivial center, so the inertia subgroup is not contained in the center of $ G $, and so $ \OL[\lambda(G)]^{G} $ is not a Hopf order in $ H_{\lambda} $. 
\\ \\
Next we suppose that $ \A_{\lambda} $ is a Hopf order in $ H_{\lambda} $ properly containing $ \OL[\lambda(G)]^{G} $. Then $ \OL \otimes_{\OK} \A_{\lambda} $ is a Hopf order in $ L \otimes_{K} L[\lambda(G)]^{G} = L[\lambda(G)] $ properly containing $ \OL[\lambda(G)] $. But $ |G| = qr $ and $ p \ndivides qr $, so $ L[\lambda(G)] $ has no Hopf orders apart from $ \OL[\lambda(G)] $, by \cite[Corollary 20.3]{ChTwe}. Therefore $ \A_{\lambda} $ is not a Hopf order in $ H_{\lambda} $. 
\end{example}

This example illustrates that Theorem \ref{thm_main_thm_intro} has the potential to extend the scope of  Childs' Theorem: if $ L/K $ is a finite separable extension of local or global fields in characteristic zero and $ H $ is a Hopf algebra giving a Hopf-Galois structure on the extension then $ \OL $ is locally free over its associated orders in both $ H $ and $ H^{\prime} $ if either of these is a Hopf order.    
\\ \\
The question of whether $ \AH $ is a maximal order in $ H $ is also of interest, and this is the subject of our final result. We shall require $ H $ to be a separable $ K $ algebra, so we first note a sufficient condition for this to occur:

\begin{lemma}\label{lem_H_separable}
Suppose that the characteristic of $ K $ does not divide $ [E:K] $. Then $ H $ is a separable $ K $-algebra. 
\end{lemma}
\begin{proof}
Recall from the theorem of Greither and Pareigis that $ H = E[N]^{G} $ for some regular subgroup $ N $ of $ \perm{X} $ normalized by $ \lambda(G) $ and that we have $ E \otimes_{K} E[N]^{G} = E[N] $. Since $ N $ is regular, we have $ |N|=[L:K] $, which divides $ [E:K] $, so the characteristic of $ E $ does not divide $ |N| $, and so $ E \otimes_{K} E[N]^{G} = E[N] $ is a separable $ E $-algebra. By \cite[Chapter II, Corollary 1.10]{Demeyer_Ingraham}, this implies that $ E[N]^{G} $ is a separable $ K $-algebra provided that  $ K $ is a $ K $-direct summand of $ E $. The map $ E \rightarrow K $ defined by
\[ x \mapsto \frac{1}{[E:K]}\mbox{Tr}_{E/K}(x) \]
is surjective, $ K $-linear, and split by the inclusion $ K \hookrightarrow E $, so $ K $ indeed occurs as a $ K $-direct summand of $ E $, and so $ E[N]^{G} $ is a separable $ K $-algebra.
\end{proof}

\begin{prop} \label{prop_maximal_orders}
Suppose that the characteristic of $ K $ does not divide $ [E:K] $. Let $ \B $ be a fractional ideal of $ L $, and let $ \AH $, $ \AH^{\prime} $ be the associated orders of $ \B $ in $ H $, $ H^{\prime} $ respectively. Then $ \AH $ is a maximal order in $ H $ if and only if $ \AH^{\prime} $ is a maximal order in $ H^{\prime} $. 
\end{prop}
\begin{proof}
Since the characteristic of $ K $ does not divide $ [E:K] $, Lemma \ref{lem_H_separable} implies that$ H $ is a separable $ K $-algebra. Thus an $ \OK $-order $ \A $ in $ H $ is maximal if and only if for each prime $ \p $ of $ \OK $ the order $ \OKp \otimes_{\OK} \A $ is a maximal $ \OKp $-order in the separable $ K_{\p} $-algebra $ K_{\p} \otimes_{K} H $ \cite[Corollaries 11.2 and 11.5]{Reiner_MO}. Therefore we may reduce to the case in which $ K $ is a local field. 
\\ \\
Suppose that $ \AH^{\prime} $ is a maximal order but $ \AH $ is not. Since $ H $ is a separable $ K $-algebra, by \cite[Corollary 10.4]{Reiner_MO} $ \AH $ is properly contained in some maximal order in $ H $, say $ {\mathfrak M} $. Since $ L $ is a free $ H^{\prime} $-module of rank $ 1 $, \cite[Theorem 18.10]{Reiner_MO} implies that $ \B $ is a free $ \AH^{\prime} $-module of rank $ 1 $. Let $ x \in \B $ generate $ \B $ as an $ \AH^{\prime} $ module; by Theorem \ref{thm_main_thm_intro} it also generates $ \B $ as an $ \AH $-module. Now let $ \Delta = {\mathfrak M} \cdot x $. This is a full $ \OK $-lattice in $ L $ which properly contains $ \B $, and it is a free $ {\mathfrak M} $-module of rank $ 1 $, generated by $ x $. Let $ \AH^{\prime}(\Delta) $ denote the associated order of $ \Delta $ in $ H^{\prime} $. Similarly to the proof of Theorem \ref{thm_main_thm_intro}, for each $ \mu \in {\mathfrak M} $ define $ z_{\mu} \in H^{\prime} $ by $ z_{\mu} \cdot x = \mu \cdot x $; we find that $ \AH^{\prime}(\Delta) = \{ z_{\mu} \mid \mu \in {\mathfrak M} \} $. But $ \AH^{\prime}(\Delta) \supsetneq \AH^{\prime} $ since
\[ \Delta = \AH^{\prime}(\Delta) \cdot x \supsetneq \AH^{\prime} \cdot x = \B. \]
This contradicts the assumption that $ \AH^{\prime}(\B) $ is a maximal order in $ H $. Therefore $ \AH $ is a maximal order in $ H $. 
\\ \\
Since $ (H^{\prime})^{\prime} = H $, the converse statement follows by interchanging the roles of $ H, H^{\prime} $ in the argument above.
\end{proof}

\bibliographystyle{amsplain}

\end{document}